\documentclass[12pt,a4paper]{article}

\usepackage[utf8]{inputenc}
\usepackage[english]{babel}
\usepackage{amsmath,amssymb,amsfonts,amsthm}
\usepackage[top=2.5cm,left=2.5cm,right=2.5cm,bottom=3cm]{geometry}
\usepackage{hyperref}
\hypersetup{hidelinks}

\newtheorem{theorem}{Theorem}
\newtheorem{proposition}[theorem]{Proposition}
\newtheorem{corollary}[theorem]{Corollary}

\theoremstyle{remark}
\newtheorem{remark}[theorem]{Remark}

\DeclareMathOperator{\SU}{SU}
\DeclareMathOperator{\U}{U}
\DeclareMathOperator{\PU}{PU}
\DeclareMathOperator{\Sp}{Sp}
\newcommand{\A}{\mathbb A}
\newcommand{\C}{\mathbb C}
\newcommand{\Q}{\mathbb Q}
\newcommand{\R}{\mathbb R}
\newcommand{\Z}{\mathbb Z}
\newcommand{\g}{\mathfrak g}

\title{Residual finiteness and cuspidal cohomology\\
of Picard modular surfaces}
\author{Richard M. Hill\\
  \small University College London\\
  \small\texttt{r.m.hill@ucl.ac.uk}\\
  \small ORCID: 0000-0001-8528-1486}
\date{}

\begin{document}
\maketitle

\begin{abstract}
We prove that, for every non-uniform arithmetic lattice in $\SU(2,1)$, its
inverse images in the universal cover and in all connected finite covers are
residually finite.
The key new input is that every commensurability class of
such lattices contains a congruence arithmetic lattice $\Gamma$
for which
\[
  H^1_{\mathrm{cusp}}(\Gamma\backslash\mathbb B^2,\C)\ne0.
\]
In particular, the first inner cohomology of this ball quotient is non-zero.
The proof uses Rogawski's endoscopic classification for $\U(3)$.
A cohomological criterion proved
previously by the author then gives the residual-finiteness result.
Residual
finiteness also yields multiplier systems of arbitrary denominator on suitable
finite-index subgroups.


\end{abstract}

\section{Introduction}

We prove that, for every non-uniform
arithmetic lattice in $\SU(2,1)$, its inverse images in the universal cover
and in all connected finite covers are residually finite; this is
Corollary~\ref{cor:rf} below.

The route to this result is cohomological.
Let $\Gamma$ be such a lattice and
put
\[
  X_\Gamma=\Gamma\backslash\mathbb B^2,
\]
where $\mathbb B^2$ is the complex hyperbolic plane.
We write
$H^j_{\mathrm{cusp}}(X_\Gamma,\C)$ for the space of harmonic cuspidal
differential $j$-forms on $X_\Gamma$; its relation with ordinary cohomology is
recalled in Section~\ref{sec:automorphic-cohomology}.
The inner cohomology is
defined by
\[
 H^j_!(X_\Gamma,\C)
 =\operatorname{im}\bigl(H^j_c(X_\Gamma,\C)
                  \longrightarrow H^j(X_\Gamma,\C)\bigr).
\]
The author's criterion, recalled in Theorem~\ref{thm:hill-criterion}, states
that if the commensurability class of a non-uniform arithmetic lattice in
$\SU(d,1)$ contains a lattice with non-zero first inner cohomology, then
the inverse image of the original lattice in every connected cover of
$\SU(d,1)$ is residually finite.
Thus Corollary~\ref{cor:rf} follows
from the following automorphic non-vanishing theorem, which is the new input
of this paper.

\begin{theorem}
\label{thm:main}
Let $\Gamma$ be a non-uniform arithmetic lattice in $\SU(2,1)$.
There is a
congruence arithmetic lattice $\Gamma'$ in the commensurability
class of $\Gamma$ such that
\[
 H^1_{\mathrm{cusp}}(X_{\Gamma'},\C)\ne0.
\]
In particular, $H^1_!(X_{\Gamma'},\C)\ne0$.
\end{theorem}

\begin{corollary}
\label{cor:rf}
Let $\Gamma$ be a non-uniform arithmetic lattice in $\SU(2,1)$.
Its inverse
images in the universal cover and in every connected finite cover of
$\SU(2,1)$ are residually finite.
\end{corollary}

The corresponding result for cocompact arithmetic lattices of the first kind
was proved independently, by different methods, by the author
\cite[Theorems~1 and~3]{HillFractional} and Stover--Toledo
\cite[Theorem~1.1]{StoverToledo}.
For non-uniform lattices, earlier work of Stover--Toledo treated the
Eisenstein Picard commensurability class over $\Q(\sqrt{-3})$
\cite[Theorems~1.1 and~8.2]{StoverToledoSurfaces}.
Corollary~\ref{cor:rf} extends this to all imaginary quadratic fields.

There is also a consequence for fractional-weight modular forms.
Let
$j(\gamma,z)$ be the standard automorphy factor for the action of
$\SU(2,1)$ on the ball.
A multiplier system of weight $1/n$ on a subgroup
$\Delta$ is an automorphy factor
$\ell\colon\Delta\times\mathbb B^2\to\C^\times$ satisfying the cocycle
identity and such that its $n$th power is
$\chi(\gamma)j(\gamma,z)$ for some character
$\chi:\Delta\to\C^\times$ with finite image.
A holomorphic function $f$
satisfying
\[
 f(\gamma z)=\ell(\gamma,z)^m f(z)
\]
and the usual holomorphy condition at the cusps\footnote{For $\SU(2,1)$,
this condition is redundant by Koecher's principle; see
\cite[Section~1.7.4]{deShalitGoren}.
For fractional weight, one applies the integral-weight statement to a
power of $f$ that clears the denominator and kills the finite-order character.}
is then a modular form of
weight $m/n$.

\begin{corollary}
\label{cor:fractional-weights}
Let $\Gamma$ be a non-uniform arithmetic lattice in $\SU(2,1)$.
For every
positive integer $n$, there is a subgroup $\Delta_n\subset\Gamma$ of finite
index that admits a multiplier system of weight $1/n$.
\end{corollary}

The level $\Delta_n$ furnished by this argument need not be a congruence
subgroup.
After passing, if necessary, to a further finite-index subgroup,
non-zero fractional-weight modular forms may be constructed from the
multiplier system, for example as Eisenstein series of sufficiently large
weight.
Jalal constructed an explicit half-integral-weight multiplier system on
$\SU(2,1)$ using local cocycles for the double cover and their splittings
over compact open subgroups \cite{Jalal}.
Freitag and the author give explicit examples of weight $1/3$ forms for
certain congruence subgroups of the Eisenstein Picard modular group, together
with a method for determining the levels at which such forms occur
\cite{FreitagHill}.

Deligne's theorem provides a contrast: for $g\geq2$, the inverse image of
$\Sp_{2g}(\Z)$ in the universal cover of $\Sp_{2g}(\R)$ is not residually
finite, nor is its inverse image in any connected finite cover of degree
greater than two \cite{Deligne}.
We discuss this comparison and related restrictions on rational weights in
Section~\ref{sec:residual-finiteness}.

\subsection{Automorphic strategy}

We find the inner-cohomology class asserted in Theorem~\ref{thm:main} among
cuspidal automorphic forms.
The automorphic description of cohomology, reviewed in
Section~\ref{sec:automorphic-cohomology}, turns the problem into a concrete
search for a cuspidal representation
$\pi=\pi_\infty\otimes\pi_f$.
The real factor $\pi_\infty$ must contribute
to cohomology in degree one.
Once such a representation has been found, the
finite factor $\pi_f$ has fixed vectors for some compact open subgroup, and
that subgroup supplies the required congruence level.
Kuga's formula and
Hodge theory identify the resulting classes with harmonic cusp forms
\cite[p.~50, equation~(1.5.1)]{Harris}
\cite[Chapter~II, Proposition~3.1]{BorelWallach}.
Theorem~\ref{thm:borel}
below identifies these classes with a subspace of inner cohomology.

Let $E/\Q$ be the imaginary quadratic extension determining the commensurability
class of $\Gamma$.
Here and below, $\U(3)$ denotes the quasi-split unitary
group in three variables over $\Q$ attached to $E/\Q$; its group of real
points is $\U(2,1)$.
We follow the notation $\U(3)$ used in the cited
automorphic sources, while $\U(2,1)$ records the signature at the real place;
see \cite[p.~904, Section~2.2]{Marshall}.

Theorem~\ref{thm:main} will be proved by constructing a cuspidal automorphic
representation of $\U(3)$ with the required real component.
Rogawski's
classification places the two degree-one cohomological representations in
endoscopic packets attached to the simpler group $H$ described in
Section~\ref{sec:endoscopic-packets}
\cite[p.~906, Section~3.3]{Marshall}.
Endoscopic
transfer does not produce a single representation of $\U(3)$: it produces a
collection, called a packet, of possible local and global representations.
Starting with a suitable automorphic character of $H$, we
will choose from its packet a member having the required cohomological
component at infinity and a supercuspidal component at a finite place; see
\cite[p.~903 and pp.~906--907, Sections~3.3--3.4]{Marshall} and
\cite[Theorem~13.3.2, Section~14, and Proposition~15.2.1]{Rogawski}.

The two local choices serve different purposes.
The real component supplies
degree-one cohomology.
The supercuspidal finite component rules out the
residual spectrum and therefore makes the global representation cuspidal.
The multiplicity formula for the packet lets us make both choices while still
obtaining a representation that occurs automorphically.
This refines the
packet construction in Marshall's study of $L^2$-cohomology growth
\cite[pp.~909--910, Section~4.2]{Marshall}.

\section{Degree-one cohomology of Picard modular surfaces}
\label{sec:automorphic-cohomology}

Let $E/\Q$ be an imaginary quadratic extension.
Up to $\Q$-isomorphism, there
is a unique special unitary group arising from an isotropic Hermitian form of
signature $(2,1)$ in three variables over $E$; it is quasi-split.
Moreover,
commensurability classes of
non-uniform arithmetic lattices in $\SU(2,1)$ are in bijection with imaginary
quadratic fields; see \cite[Section~3]{Stover}.
It is therefore enough to
prove automorphic non-vanishing for the quasi-split group attached to an
arbitrary $E$.

Write $\A$ for the ring of ad\`eles of $\Q$ and $\A_f$ for its finite part.
Let $G=\U(3)$ be this quasi-split unitary group, and choose a maximal compact
subgroup $K_\infty\subset G(\R)$.
Then
\[
 K_\infty\simeq\mathrm U(2)\times\mathrm U(1).
\]
Here the groups on the right are the compact real unitary groups; this
notation is distinct from the global unitary groups used below.
The quotient
$G(\R)/K_\infty$ is the complex two-ball.
Let $Z$ denote the centre of $G$.
Write $G^{\mathrm{ad}}=G/Z$ for its adjoint quotient.
We work throughout with automorphic forms on which $Z(\A)$ acts trivially.
For a compact open subgroup $K_f\subset G(\A_f)$, put
\[
 X_{K_f}=G(\Q)\backslash G(\A)/(K_\infty K_f Z(\A)).
\]
The space $X_{K_f}$ is a
finite disjoint union of quotients of the complex ball by congruence
arithmetic lattices in the adjoint group $G^{\mathrm{ad}}(\R)=\PU(2,1)$.
Choosing $K_f$ sufficiently small makes these lattices torsion-free and the
quotients manifolds.

We next recall how representations of $G(\R)$ detect differential forms on
these quotients.
Write
\[
 \g=\operatorname{Lie}(G(\R))\otimes_\R\C,
 \qquad
\mathfrak k=\operatorname{Lie}(K_\infty)\otimes_\R\C.
\]
For a smooth representation $V$ of $G(\R)$, its relative Lie algebra
cohomology $H^j(\g,K_\infty;V)$ is the cohomology of the complex
\[
 C^j(\g,K_\infty;V)
 =\operatorname{Hom}_{K_\infty}
   \bigl(\textstyle\bigwedge^j(\g/\mathfrak k),V\bigr),
\]
with the standard Lie algebra differential.
This complex is the
representation-theoretic model for differential forms on
$G(\R)/K_\infty$.
In particular, saying that $V$ is cohomological in degree
$j$ means that $H^j(\g,K_\infty;V)$ is non-zero.

Let $\mathcal C_{\mathrm{cusp}}(G,1)$ denote the space of smooth,
$K_\infty$-finite
cuspidal automorphic forms on $G(\A)$ with trivial central character.
Its
$L^2$-completion decomposes discretely into irreducible unitary
representations.
Denote by $\mathcal A_{\mathrm{cusp}}(G,1)$ the set of
representations occurring in this decomposition and by
$m_{\mathrm{cusp}}(\pi)$ the multiplicity of $\pi$.
Every such representation
factors as
$\pi=\pi_\infty\otimes\pi_f$.
Passing to $K_f$-fixed vectors selects the
finite-dimensional space $\pi_f^{K_f}$, while relative Lie algebra
cohomology is applied to $\pi_\infty$.
Decomposing the cuspidal spectrum
therefore gives the cuspidal Matsushima formula
\begin{equation}
\label{eq:matsushima}
 H^j(\g,K_\infty;
       \mathcal C_{\mathrm{cusp}}(G,1)^{K_f})
 \simeq
 \bigoplus_{\pi\in\mathcal A_{\mathrm{cusp}}(G,1)}
 \left(
 H^j(\g,K_\infty;\pi_\infty)\otimes\pi_f^{K_f}
 \right)^{\oplus m_{\mathrm{cusp}}(\pi)}.
\end{equation}
Each summand is unitary, so the Hodge theorem for relative Lie algebra
cohomology identifies its cohomology with its harmonic cochains
\cite[Chapter~II, Proposition~3.1]{BorelWallach}.
Kuga's formula identifies
the corresponding cochain Laplacian with the geometric Laplacian
\cite[p.~50, equation~(1.5.1)]{Harris}.
Thus the left-hand side of
\eqref{eq:matsushima} is precisely the space of harmonic cusp forms that
Harris denotes by $H^j_{\mathrm{cusp}}(X_{K_f},\C)$.

\begin{theorem}[Borel]
\label{thm:borel}
Let $\Lambda$ be a torsion-free arithmetic lattice in $\SU(2,1)$ or
$\PU(2,1)$, and put $Y=\Lambda\backslash\mathbb B^2$.
Sending a harmonic
cuspidal differential form to its de Rham cohomology class defines an
injective map
\[
 H^j_{\mathrm{cusp}}(Y,\C)\longrightarrow H^j(Y,\C)
\]
whose image is contained in the inner cohomology
\[
 H^j_!(Y,\C)
 =\operatorname{im}\bigl(H^j_c(Y,\C)\longrightarrow H^j(Y,\C)\bigr).
\]
\end{theorem}

\begin{proof}
The assertion for rapidly decreasing harmonic forms is proved in
\cite[Sections~5.3 and~5.5]{BorelStableII}, where Section~5.5 applies the
result to harmonic cusp forms.
The underlying analytic injectivity argument
also appears in \cite[Proposition~2.5 and Remark~3.7]{BorelStableI}.
See
\cite[p.~51, Theorem~1.5.2]{Harris} for the formulation used here.
\end{proof}

Applying Theorem~\ref{thm:borel} componentwise shows that the canonical map
from $H^j_{\mathrm{cusp}}(X_{K_f},\C)$ to ordinary cohomology is injective
and has image in inner cohomology.

The formula~\eqref{eq:matsushima} gives a convenient test for non-vanishing.
It
is enough to find a cuspidal representation that occurs with positive
multiplicity and whose real component has non-zero degree-one cohomology.
Indeed, the finite component of any such representation is smooth and hence
has a non-zero vector fixed by some compact open subgroup.

Following \cite[p.~906, Section~3.3]{Marshall}, let $J^+$ and $J^-$ denote
the two irreducible unitary representations of $G(\R)$ with trivial central
character and non-zero relative Lie algebra cohomology in degree one.
They
are non-tempered and are distinguished by the Hodge types $(1,0)$ and
$(0,1)$, respectively.
Equivalently, they are the non-tempered members of
the real endoscopic packets $\{J^+,D^-\}$ and $\{J^-,D^+\}$, whose other
members are discrete series.
Moreover,
\[
 \dim H^1(\g,K_\infty;J^+)
 =\dim H^1(\g,K_\infty;J^-)=1.
\]
This is the degree-one part of Rogawski's classification
\cite[Proposition~15.2.1]{Rogawski}.
We will construct a cuspidal $\pi$
with $\pi_\infty=J^+$.

\section{Endoscopic packets}
\label{sec:endoscopic-packets}

\subsection{The endoscopic group and transfer}

Every one-dimensional unitary automorphic representation of the endoscopic
group $H=\U(2)\times\U(1)$ gives rise to a packet for $G=\U(3)$.
Here $\U(2)$ is the quasi-split unitary group of a hyperbolic Hermitian plane
over $E/\Q$ (thus $\U(2)(\R)\simeq\U(1,1)$), while $\U(1)$ is the norm-one
torus associated with $E/\Q$.
We use Rogawski's classification in the range in which the multiplicity
formula is justified by Flicker's work.
The qualification at $2$
and its role in the proof are made explicit below; see
\cite[p.~904, Remark]{Marshall}.
Formally, the transfer is specified by an embedding of $L$-groups
${}^LH\to{}^LG$.
The group $H$, the chosen embedding, and the two elementary maps
needed below are described in
\cite[p.~905, Section~2.4]{Marshall}.
Explicitly,
\[
 \det\nolimits_0:H\longrightarrow\U(1),\qquad
 \lambda:H\longrightarrow\U(1)
\]
are the determinant on the $\U(2)$ factor and projection onto the $\U(1)$
factor, respectively.

\subsection{Characters on the endoscopic group}

The $L$-group embedding requires an auxiliary Hecke character.
Fix a Hecke
character $\mu$ of $E^\times\backslash\A_E^\times$ whose restriction to
$\Q^\times\backslash\A^\times$ is the quadratic character associated with
$E/\Q$, as in \cite[p.~904, Section~2.1]{Marshall}.
When $\mu$ is evaluated on a norm-one id\`ele, we use the same symbol for its
restriction.
To describe the automorphic characters of $H$ that we will
transfer, define the groups
\[
 E^1=\{x\in E^\times:N_{E/\Q}(x)=1\},
 \qquad
 \A_E^1=\{x\in\A_E^\times:N_{E/\Q}(x)=1\}.
\]
These are the rational and adelic points of the norm-one torus
\[
 \U(1)=\ker\bigl(\operatorname{Res}_{E/\Q}\mathbb G_m
                 \xrightarrow{N_{E/\Q}}\mathbb G_m\bigr),
\]
and
\[
 I_E^1=E^1\backslash\A_E^1
\]
is its id\`ele class group.

Let $\theta:I_E^1\to\C^\times$ be a unitary character.
It determines the
one-dimensional automorphic representation
\begin{equation}
\label{eq:xi}
 \xi_\theta=(\theta\circ\det\nolimits_0)\otimes
       \bigl((\theta^{-2}\mu^{-1})\circ\lambda\bigr)
\end{equation}
of $H(\A)$; this is the parametrization in
\cite[p.~908, immediately after (6)]{Marshall}.
The factor $\theta^{-2}\mu^{-1}$ ensures that the restriction of
$\xi_\theta$ to the diagonally embedded centre is $\mu^{-1}$.
Our normalization of transfer multiplies this by $\mu$, so the transferred
packet has trivial central character.
To obtain the required cohomological member of the real packet, we prescribe
the infinity type of $\theta$ as follows.
Let $t\in\Z$ be defined by
$\mu_\infty(z)=(z/\overline z)^{t+1/2}$ for $z\in\C^\times$, where the
right-hand side means $(z/|z|)^{2t+1}$.
Choose
$\theta_\infty(z)=z^{-t-1}$ on $\U(1)(\R)$.

We must check that this prescribed $\theta_\infty$ is the infinity component
of a global character of $I_E^1$.
Writing $\theta=\theta_\infty\theta_f$, the condition that $\theta$ be
trivial on the diagonally embedded group $E^1$ is
\[
 \theta_f(a)=\theta_\infty(a)^{-1}\qquad(a\in E^1).
\]
Thus the finite component must cancel the infinity component on rational
points; it cannot in general be chosen unramified everywhere.

To see that such a choice is possible, take a compact open subgroup
$K\subset\A_{E,f}^1$ contained in the finite integral points of the
norm-one torus.
Then $E^1\cap K$ consists of norm-one units of $E$, which are roots of
unity because $E$ is imaginary quadratic.
There are only finitely many of these, so we can shrink $K$ until
$E^1\cap K=\{1\}$.
Consequently, prescribing $\theta_\infty$ on the real factor and the
trivial character on $K$ defines a character on
$E^1(\U(1)(\R)K)$ that is trivial on $E^1$.
The quotient $\A_E^1/E^1(\U(1)(\R)K)$ is finite, and a unitary character
of a subgroup of finite index in an abelian group extends to the whole
group.
This gives the required global $\theta$ with finite component trivial on
$K$.
Shrinking $K$ amounts to allowing a larger finite conductor.
Marshall uses this freedom to count characters with the prescribed
infinity type in his sets $\Theta(\mathfrak n)$
\cite[p.~909, Section~4.2]{Marshall}; here we need only one such character.

\subsection{Local and global packets}

For each place $v$ of $\Q$, write $\theta_v$, $\mu_v$, and
$\xi_{\theta,v}$ for the corresponding local components.
Endoscopic transfer
associates to $\xi_{\theta,v}$ a finite local packet
$\Pi_v(\xi_{\theta,v})$ of irreducible representations of $G(\Q_v)$.
The associated global packet is the set of restricted tensor products
\[
 \Pi(\xi_\theta)
 =\left\{\bigotimes'_v\pi_v:
        \begin{array}{l}
        \pi_v\in\Pi_v(\xi_{\theta,v})\text{ for every }v,\\
        \pi_v\text{ is unramified at all but finitely many finite places}
        \end{array}\right\}.
\]
This is the packet construction of
\cite[pp.~906--907, Sections~3.1--3.4]{Marshall};
see also \cite[Theorem~13.3.2 and Section~14]{Rogawski}.

We now describe the local packets that enter the proof.

\smallskip
\noindent\emph{The real place.}
For the character $\theta$ chosen above, substitution in
\eqref{eq:xi} gives
$\xi_{\theta,\infty}=(\det\nolimits_0)^{-t-1}\lambda$, and the archimedean packet
calculation in \cite[p.~907, Section~3.3]{Marshall} then gives
\begin{equation}
\label{eq:infinity-packet}
 \Pi_\infty(\xi_{\theta,\infty})=\{J^+,D^-\},
\end{equation}
where $J^+$ is the degree-one cohomological representation introduced above
and $D^-$ is a discrete-series representation.
In Marshall's notation,
$J^+$ is the non-tempered member $\pi_\infty^n$ of this packet, whereas
$D^-$ is the tempered member $\pi_\infty^s$.

\smallskip
\noindent\emph{Split finite places.}
If $v$ splits in $E/\Q$, the local packet is the singleton
\[
 \Pi_v(\xi_{\theta,v})=\{\pi_v^n(\xi_{\theta,v})\},
\]
whose sole member is the representation induced from the corresponding
character of the Levi subgroup $H(\Q_v)$; see
\cite[p.~906, Section~3.1]{Marshall}.
When $\theta_v$ and $\mu_v$ are
unramified, this representation is unramified.

\smallskip
\noindent\emph{Nonsplit finite places.}
Suppose that $v$ is a nonsplit finite place of $E/\Q$.
The local packet has
two members:
\begin{equation}
\label{eq:local-packet}
 \Pi_v(\xi_{\theta,v})
 =\{\pi_v^n(\xi_{\theta,v}),\pi_v^s(\xi_{\theta,v})\}.
\end{equation}
The superscripts distinguish the non-tempered member $\pi_v^n$ from the
supercuspidal member $\pi_v^s$; see
\cite[p.~906, Section~3.2]{Marshall}.
When all the local data---in
particular, $E/\Q$, $\theta_v$, and $\mu_v$---are unramified, $\pi_v^n$ is
the unramified (equivalently, spherical) member, whereas $\pi_v^s$ is
supercuspidal and hence ramified.

\subsection{The automorphic multiplicity}

Rogawski's multiplicity formula determines which members of the packet
occur in the discrete automorphic spectrum.
If the representation
$\pi=\bigotimes'_v\pi_v$ belongs to $\Pi(\xi_\theta)$, let the integer
$n(\pi)$ be the number of places at which the local packet has two members
and $\pi_v=\pi_v^s(\xi_{\theta,v})$.
This number is finite since the supercuspidal representations $\pi_v^s$ are ramified.
Let
$m_{\mathrm{disc}}(\pi)$ denote the multiplicity of $\pi$ in the discrete
automorphic spectrum.
Suppose, in addition, that $\pi_2=\pi_2^n(\xi_{\theta,2})$ if $2$ is
nonsplit in $E/\Q$.
Thus $\pi_2$ is a constituent of a parabolically induced representation.
This is
precisely the case in which Flicker's result
permits one to use Rogawski's multiplicity formula despite the general
qualification at $2$; see \cite[p.~904, Remark]{Marshall}.
Under this hypothesis the formula is
\begin{equation}
\label{eq:multiplicity}
 m_{\mathrm{disc}}(\pi)=\frac12\left(1+\epsilon(\xi_\theta,\mu) \cdot
                    (-1)^{n(\pi)}\right),
\end{equation}
where the global packet sign $\epsilon(\xi_\theta,\mu)$ is either $1$ or
$-1$.
Among packet members satisfying the stated hypothesis at $2$, precisely
those with $(-1)^{n(\pi)}=\epsilon(\xi_\theta,\mu)$ occur in the discrete
automorphic spectrum, each with multiplicity one; see
\cite[p.~907, Section~3.4]{Marshall}, \cite[p.~218]{Flicker}, and
\cite{RogawskiMultiplicity}.

\section{A cuspidal packet member}

The following construction refines Marshall's lower-bound argument by
requiring his set $I$ of supercuspidal choices to be nonempty
\cite[p.~909, Section~4.2]{Marshall}.

\begin{proposition}
\label{prop:cusp}
For every imaginary quadratic field $E$, the quasi-split group $G=\U(3)$
attached to $E/\Q$ has a cuspidal automorphic representation $\pi$ with
trivial central character and $\pi_\infty=J^+$.
\end{proposition}

\begin{proof}
Choose $\theta$ as above and put $\xi=\xi_\theta$.
By \eqref{eq:infinity-packet}, we may take $\pi_\infty=J^+$.
We choose the finite components to obtain multiplicity one and at least one
supercuspidal component.

There are infinitely many odd primes that are inert in $E/\Q$.
At any such prime $v$, the two-member packet \eqref{eq:local-packet} is available.
Choose the supercuspidal representation $\pi_v^s(\xi_v)$ at one inert prime.
If a single supercuspidal choice gives the wrong parity in
\eqref{eq:multiplicity}, make a second supercuspidal choice at another inert odd
prime.
Choose the
non-tempered member $\pi_v^n(\xi_v)$ at every other nonsplit finite place, in
particular at the prime $v=2$.
At split places there is no choice.
At almost every place all the data
are unramified and we have chosen the unramified member, so these local
choices form a well-defined restricted tensor product $\pi$.
The set of supercuspidal choices is finite
and nonempty, and its cardinality has the parity for which the
multiplicity formula gives $m_{\mathrm{disc}}(\pi)=1$.
Hence the resulting packet member $\pi$ occurs in the discrete automorphic spectrum
with multiplicity one.

Since $\pi$ has a supercuspidal finite component, it is cuspidal
\cite[p.~219]{Flicker}.
Indeed, every local component of a residual representation is a subquotient
of a proper parabolic induction, whereas a supercuspidal representation is,
by definition, not such a subquotient.
\end{proof}

\begin{remark}
Proposition~\ref{prop:cusp} remains true with $J^-$ in place of $J^+$.
Indeed, the argument may be repeated using the other archimedean packet
$\{J^-,D^+\}$.
\end{remark}

\begin{proof}[Proof of Theorem~\ref{thm:main}]
Fix the imaginary quadratic field $E$ determining the commensurability class
of $\Gamma$, and let the representation $\pi$ be supplied by
Proposition~\ref{prop:cusp}.
Since
$\pi_f$ is smooth, there is a compact open subgroup $K_f\subset G(\A_f)$ for
which $\pi_f^{K_f}\ne0$.
After shrinking $K_f$, we may assume that the
associated arithmetic groups are torsion-free; the chosen vector remains fixed when
the level is made smaller.
Since $\pi$ is cuspidal with multiplicity one, $\pi_\infty=J^+$, and
$\pi_f^{K_f}\ne0$, its degree-one summand in \eqref{eq:matsushima} is non-zero.
Thus
\[
 H^1_{\mathrm{cusp}}(X_{K_f},\C)\ne0.
\]
The adelic locally symmetric space $X_{K_f}$ is a finite disjoint union of
arithmetic ball quotients.
Both cuspidal and inner cohomology decompose as
direct sums over the connected components.
In particular, some component
$X_{\overline\Lambda}=\overline\Lambda\backslash\mathbb B^2$, with
$\overline\Lambda\subset G^{\mathrm{ad}}(\Q)$ a congruence arithmetic
lattice, satisfies
\[
 H^1_{\mathrm{cusp}}(X_{\overline\Lambda},\C)\ne0.
\]
Let $\overline K_f$ be the image of $K_f$ in
$G^{\mathrm{ad}}(\A_f)$, and choose an adelic representative
$\overline g_f$ of this component.
Thus $\overline\Lambda$ is defined by
the conjugated finite level
$\overline g_f\overline K_f\overline g_f^{-1}$.

Let $G^{\mathrm{der}}$ be the derived subgroup of $G$, and consider the
central isogeny
\[
 q:G^{\mathrm{der}}\longrightarrow G^{\mathrm{ad}},
 \qquad
 q_{\R}:\SU(2,1)\longrightarrow\PU(2,1).
\]
Choose a compact open subgroup $K_f^{\mathrm{der}}$ of
$G^{\mathrm{der}}(\A_f)$ whose image under $q$ lies in
$\overline g_f\overline K_f\overline g_f^{-1}$.
After shrinking
$K_f^{\mathrm{der}}$ if necessary, set
\[
 \Gamma'=G^{\mathrm{der}}(\Q)\cap K_f^{\mathrm{der}}.
\]
Then $\Gamma'$ is a torsion-free congruence arithmetic lattice and
$q(\Gamma')\subset\overline\Lambda$.
The image of an arithmetic subgroup
under an isogeny is arithmetic \cite[Section~8.9]{BorelArithmetic}; hence
$q(\Gamma')$ and $\overline\Lambda$ are lattices in the same real group, so
the inclusion has finite index.
Consequently, the induced map
\[
 p:X_{\Gamma'}\longrightarrow X_{\overline\Lambda}
\]
is a finite covering.
Its pullback on ordinary cohomology is injective by
the transfer map.
Pullback also preserves cuspidality.
Indeed, the central
isogeny identifies the unipotent radicals of corresponding rational
parabolic subgroups, so formation of constant terms commutes with pullback.
Consequently
\[
 H^1_{\mathrm{cusp}}(X_{\Gamma'},\C)\ne0.
\]
Applying Theorem~\ref{thm:borel} gives $H^1_!(X_{\Gamma'},\C)\ne0$.
Finally, the classification of
non-uniform arithmetic lattices recalled above shows that $\Gamma'$ is
commensurable with the original lattice $\Gamma$.
\end{proof}

\begin{remark}
The proof of Theorem~\ref{thm:main} produces inner cohomology at some
congruence level, but does not attempt to determine the smallest such level.
It also explains why the construction lies naturally in the endoscopic part
of the automorphic spectrum: the two
degree-one cohomological representations of $\U(2,1)$ are non-tempered
\cite[p.~906, Section~3.3]{Marshall}.
\end{remark}

\section{Residual finiteness}
\label{sec:residual-finiteness}

We recall the cohomological criterion that converts the non-vanishing in
Theorem~\ref{thm:main} into residual finiteness.

\begin{theorem}[{\cite[Theorem~2]{Hill}}]
\label{thm:hill-criterion}
Let $d\geq2$, and let $\Gamma$ be a non-uniform arithmetic lattice in
$\SU(d,1)$.
Suppose that there is an arithmetic lattice $\Gamma'$ commensurable with
$\Gamma$ such that
\[
 H^1_!(\Gamma'\backslash\mathbb B^d,\C)\ne0,
\]
where $\mathbb B^d$ is the complex hyperbolic $d$-ball.
Then the inverse image of $\Gamma$ in every connected cover of $\SU(d,1)$
is residually finite.
\end{theorem}

\begin{proof}[Proof of Corollary~\ref{cor:rf}]
Theorem~\ref{thm:main} supplies an arithmetic subgroup
$\Gamma'$ commensurable with $\Gamma$ for which
$H^1_!(X_{\Gamma'},\C)\ne0$.
Theorem~\ref{thm:hill-criterion}
applies with $d=2$ and gives residual finiteness
of the inverse image of $\Gamma$ in every connected cover of $\SU(2,1)$.
\end{proof}

\begin{proof}[Proof of Corollary~\ref{cor:fractional-weights}]
A splitting of the relevant central
extension over $\Delta$---equivalently, a lift of $\Delta$ to the connected
$n$-fold cover of $\SU(2,1)$---produces a multiplier system of weight $1/n$
with $\chi$ trivial \cite[Section~2.1]{HillFractional}.
Conversely, since
$\chi$ has finite image, restriction to its finite-index kernel gives a
multiplier system with trivial character, and the corresponding central
extension splits over that subgroup.

Fix $n\ge1$, and let $\widetilde\Gamma^{(n)}$ be the inverse image of
$\Gamma$ in the connected $n$-fold cover.
Its central kernel is the finite
group $\Z/n\Z$.
By Corollary~\ref{cor:rf},
$\widetilde\Gamma^{(n)}$ is residually finite.
We may therefore choose a
finite-index normal subgroup $K$ of $\widetilde\Gamma^{(n)}$ which meets this
central kernel trivially: take the intersection of the kernels of finitely
many finite quotients separating its non-identity elements.
The projection
to $\Gamma$ identifies $K$ with a finite-index subgroup
$\Delta_n\subset\Gamma$, and its inverse gives a lift of $\Delta_n$ to the
$n$-fold cover.

The lift just constructed therefore supplies a multiplier system of weight
$1/n$ on $\Delta_n$, with trivial character $\chi$
\cite[Section~2.1]{HillFractional}.
\end{proof}

There is a useful contrast with arithmetic groups having finite congruence
kernel.
The author proved that if a simple, algebraically simply connected group
over $\Q$ has positive real rank and finite congruence kernel, then every
one of its arithmetic subgroups admits an extension by a finite abelian
group which is not residually finite
\cite[Theorem~1]{HillNonResidual}.
For the canonical extensions arising from
connected covering groups, however, an earlier theorem of Deligne \cite{Deligne} gives a
sharper conclusion in the basic symplectic example.
If $g\geq2$, every
finite-index subgroup of the inverse image of $\Sp_{2g}(\Z)$ in the universal
cover of $\Sp_{2g}(\R)$ contains twice the fundamental group
$2\cdot \pi_1(\Sp_{2g}(\R))$.
Hence this inverse image is not
residually finite; the same is true in every connected finite cover of degree
greater than two.

By contrast, the inverse image of $\Sp_{2g}(\Z)$ in the
metaplectic double cover is residually finite.
Indeed, classical Siegel theta
series give non-zero Siegel modular forms of half-integral weight on a
finite-index congruence subgroup of $\Sp_{2g}(\Z)$.
The square of the theta multiplier differs from the standard Siegel
automorphy factor by a finite-order character.
On the kernel of that character, the multiplier therefore defines a lift
to the metaplectic double cover.
This lift is isomorphic to a finite-index subgroup of $\Sp_{2g}(\Z)$,
so it is finitely generated, linear, and residually finite.
It has finite index in the full inverse image of $\Sp_{2g}(\Z)$.
Residual finiteness is preserved under finite extensions of finitely
generated groups, so the full inverse image is residually finite.
Deligne's theorem
shows that this degree-two exception is best possible \cite{Deligne}.

Related restrictions on rational weights hold for orthogonal groups.
For $\ell\geq3$, the author used Deligne's argument and Kneser's
theorem for spin groups to prove that every non-zero modular form of rational
weight on $\operatorname{SO}^{+}(2,\ell)$ with respect to an arithmetic
subgroup has half-integral weight
\cite[Theorem~7]{HillOrthogonal}.
Applied to Borcherds products, this weight
restriction yields congruences between vector-valued Eisenstein series and
cusp forms of half-integral weight \cite[Theorem~2]{HillOrthogonal}.
The groups considered in the current paper lie outside that finite-congruence-kernel setting,
and Corollary~\ref{cor:rf} shows that the corresponding central extensions are
residually finite.
Corollary~\ref{cor:fractional-weights} permits arbitrary denominators after passage to suitable finite-index
subgroups.

Multiplier systems of nonintegral weight also enter the Selberg trace
formulae of Ayaz and Intissar for Maass Laplacians on compact complex
hyperbolic quotients \cite[Section~3]{AyazIntissar}.
Although their analytic formulas involve a real weight parameter, their
application to a lattice quotient requires an appropriate multiplier system.
For a fixed arithmetic lattice in $\SU(2,1)$, the argument of
Freitag--Hill \cite[Theorem~4 and Proposition~2]{FreitagHill} implies that
the possible weights are rational with a common bounded denominator.
Thus the arbitrary denominators in
Corollary~\ref{cor:fractional-weights} are obtained by allowing the
finite-index subgroup to vary.

\section*{Acknowledgements}

This work was supported in part by access to ChatGPT provided through
OpenAI's ChatGPT for Academic Researchers program.
ChatGPT was used to
assist with literature review, drafting, and editing.
All AI-assisted
outputs were reviewed and verified by the author, who takes full
responsibility for the content and conclusions of this work.

\end{document}